\documentclass[11pt,a4paper]{article}
\usepackage[a4paper,top=2cm,bottom=2cm,left=2.8cm,right=2.8cm,marginparwidth=1.75cm]{geometry}

\usepackage[utf8]{inputenc}
\usepackage[english]{babel}
\usepackage{amsmath}
\usepackage{amsfonts}
\usepackage{amsthm}
\usepackage{amssymb}
\usepackage{mathtools}
\usepackage{multicol}
\usepackage{graphicx}
\usepackage{xcolor}
\graphicspath{ {pics/} }
\usepackage{caption}
\usepackage{enumerate}
\usepackage{authblk}
\usepackage{tikz}
\usepackage{float}
\usepackage[numbers, round]{natbib}
\usepackage{hyperref}

\usepackage{array} 
\usepackage{booktabs}

\definecolor{coop}{RGB}{255,220,0}
\definecolor{mixed}{RGB}{255,140,0}
\definecolor{defe}{RGB}{255,0,0}

\newtheorem{defi}{Definition}
\newtheorem{lemma}{Lemma}
\newtheorem{thm}{Theorem}
\newtheorem{coro}{Corollary}
\newtheorem{prop}{Proposition}
\newtheorem{remark}{Remark}

\theoremstyle{definition}
\newtheorem{example}{Example}

\newcommand{\T}{\Theta^{*}}
\newcommand{\Tm}{\Theta^{\mathfrak{m}}}
\newcommand{\Tp}{\Theta^{\mathfrak{p}}}
\newcommand{\Tbm}{\Theta^{\mathfrak{bm}}}
\newcommand{\TN}{\Theta^{\mathsf{NE}}}
\newcommand{\TBN}{\Theta^{\mathsf{BNE}}}
\newcommand{\M}{\mathcal{M}}

\newcommand{\diag}{\textrm{diag}}
\newcommand{\x}{\mathbf{x}}
\newcommand{\y}{\mathbf{y}}
\newcommand{\z}{\mathbf{z}}

\DeclareMathOperator{\Tr}{Tr}
\DeclareMathOperator{\rank}{rank}

\newcommand{\staricon}[3]{%
\begin{tikzpicture}[scale=1,baseline=-0.5ex,
    every node/.style={
        circle,
        draw,
        minimum size=4.8mm,
        inner sep=0pt,
        font=\scriptsize
    },
    every path/.style={thick}
]
    \node[fill=#1] (1) at (0,0) {$1$};
    \node[fill=#2] (2) at (1.4,0.4) {$2$};
    \node[fill=#3] (3) at (1.4,-0.4) {$3$};
    \draw (1)--(2);
    \draw (1)--(3);
\end{tikzpicture}%
}

\newenvironment{keywords}
{\small\noindent\textbf{Keywords:}}
{\par\bigskip}

\title{On the Structure and Stability of Boundary Mixed Steady States in Evolutionary Games on Networks}
\author[1]{Chiara Mocenni}
\author[2]{Jean Carlo Moraes\thanks{Corresponding author: jean.moraes@ufrgs.br}}
\author[2]{Matheus C. Santos}
\affil[1]{DIISM, Universit\`a di Siena, Siena, Italy}
\affil[2]{IME, Universidade Federal do Rio Grande do Sul, Porto Alegre, Brazil}
\date{\today}

\begin{document}

\maketitle

\begin{abstract}
We study steady states of evolutionary games on networks in which some players adopt pure strategies while others play mixed strategies. We refer to these configurations as boundary mixed steady states. Such states arise naturally in structured populations and have no counterpart in the classical well-mixed setting. We introduce a relaxed equilibrium notion, called boundary Nash equilibrium, in which the Nash condition is imposed only on non-pure players. In two-strategy systems, this notion characterizes boundary mixed steady states, while this correspondence breaks down in higher dimensions. The stability of these states is governed by the interaction structure among mixed players. When mixed players do not interact, the system exhibits continua of equilibria. In contrast, any nontrivial interaction generically produces instability. In particular, boundary mixed steady states that are not fully degenerate are never asymptotically stable. Degeneracies are further linked to the rank properties of the underlying interaction. These results reveal a structural instability mechanism specific to networked replicator dynamics, highlighting a qualitative gap with respect to the classical well-mixed case and showing how network topology influences the local behavior of equilibria.
\end{abstract}


\section{Introduction}

Evolutionary game theory, originally developed to model animal conflict \cite{SmPr73,smith_1982}, provides a fundamental framework to describe the evolution of strategic behavior in populations, where strategy frequencies evolve according to their relative performance \cite{hofbauer_sigmund,TAJO78}. In the classical setting of well-mixed populations, the dynamics are governed by the replicator equation, which establishes a deep connection between dynamical stability and game-theoretic concepts such as Nash equilibria and evolutionarily stable strategies; see, e.g., \cite{hofbauer_sigmund, smith_1982,TAJO78,weibull1995}. In particular, for two-strategy games the dynamics can be completely characterized, while for three-strategy games a full qualitative classification of the phase portrait is available; see \cite{Carvalho1983, Zeeman1980}.

A key limitation of the classical framework is the assumption of homogeneous mixing, whereby each individual interacts equally with the entire population. In many applications, however, interactions are structured by an underlying network, and individuals interact only with a restricted set of neighbors. This has motivated the study of evolutionary game dynamics on graphs and networks; see, for instance, \cite{JaZe15,MaWeFu14,ON2006,OhNo06}. In such settings, the structure of the interaction network plays a fundamental role in shaping the qualitative behavior of the dynamics, and classical properties of the replicator equation may no longer persist.

A systematic deterministic formulation of replicator dynamics on networked populations was introduced by Madeo and Mocenni in \cite{MaMo15}, building on earlier approaches to evolutionary dynamics on graphs; see, e.g., \cite{ON2006,OhNo06}. In this framework, each vertex of a graph represents a subpopulation whose internal strategic composition evolves according to payoff differences generated by interactions with neighboring vertices. The resulting system consists of coupled nonlinear differential equations, where the evolution at each node depends on the state of its neighbors. The classical replicator equation is recovered as a special case when the interaction graph is complete and the population is homogeneous.

Several classes of equilibria arise in this setting. Interior equilibria, in which all players adopt mixed strategies, exhibit qualitative behavior that differs significantly from the classical case. In particular, for two-strategy games, equilibria that may be asymptotically stable in the well-mixed replicator dynamics can become unstable when interactions are structured by a network \cite{MaMo15}. At the other extreme, pure steady states are typically numerous, with their number growing exponentially in the size of the network. Madeo and Mocenni \cite{MaMo15} provided a condition to test whether a given pure steady state is a Nash equilibrium, but for large graphs this approach may become computationally prohibitive. More recently, \cite{MoMo24} showed that, for suitable classes of games, graph-theoretic arguments can be used to exploit the network structure and identify pure configurations that are Nash equilibria.

Between these two extremes lie equilibria in which some players adopt pure strategies
while others play mixed strategies. Such configurations arise naturally on the boundary
of the strategy space and represent heterogeneous strategic profiles across the network.
They can also be related to models with zealotry, where a subset of agents maintains fixed
strategies and influences the dynamics of the population \cite{Kitching2025, Szolnoki2016}. From a
game-theoretic perspective, these equilibria can be interpreted as boundary configurations
in which Nash-type conditions are satisfied only on a subset of players. While interior
and pure equilibria have been studied, a systematic understanding of the
stability of such boundary equilibria is still lacking.

The goal of this paper is to analyze the existence and stability of such boundary equilibria in evolutionary games on networks. Our main result shows that whenever a nontrivial subset of players adopting mixed strategies interacts among themselves, the corresponding equilibrium is necessarily unstable. More precisely, the Jacobian matrix at the equilibrium admits a block structure whose restriction to the set of mixed players is directly related to the adjacency matrix of the induced subgraph, and it has eigenvalues with both positive and negative real parts. Then, the equilibrium exhibits both stable and unstable directions.

This result highlights a fundamental contrast with the classical replicator dynamics, where mixed Nash equilibria may be asymptotically stable. In the network setting, however, the interaction structure imposes strong constraints on stability, and mixed configurations fail to be asymptotically stable. More broadly, the analysis shows that the topology of the interaction network plays a decisive role in determining the qualitative behavior of the dynamics, beyond what is predicted by payoff structure alone.


\section{Preliminaries}

We begin by recalling the general formulation of evolutionary games on networks introduced in~\cite{MaMo15}, which allows for an arbitrary number of strategies. This general framework will be useful to introduce the notion of boundary Nash equilibrium and to establish its relation with the dynamics. We will later specialize to the case of two strategies.

Consider a finite population of $N$ players indexed by $v\in\{1,\ldots,N\}$, interacting through an undirected graph $\mathcal G$ with adjacency matrix $A=[a_{v,w}]$. Each player $v$ is associated with a mixed strategy $x_v \in \Delta_M$, where $$\Delta_M = \left\{(z_1,\ldots,z_M)\in[0,1]^M\;\left|\; \sum_{i=1}^Mz_i=1 \right.\right\}$$ 
denotes the simplex of probability vectors over $M$ pure strategies.

The global state of the system is a strategy profile given by
\[
X = (x_1,\ldots,x_N)\in \Delta_M^N.
\]
where $x_v=(x_{v,1},\ldots,x_{v,M})$ is the mixed strategy vector for player $v$.

Each player $v$ is endowed with a payoff matrix $B_v \in \mathbb{R}^{M\times M}$. Given a profile $X$, we define the aggregate strategy of the neighbors of player $v$ as
\[
k_v(X) = \sum_{w=1}^N a_{v,w} x_w.
\]
The payoff of player $v$ when playing the pure strategy $s\in\{1,\ldots,M\}$ against its neighbors is
\[
p_{v,s} (X) = e_s^\top B_v\, k_v(X),
\]
where $e_s$ denotes the $s$-th canonical basis vector. Therefore, the average payoff of $v$ when playing the mixed strategy $x_v$ is given by
\[\phi_v (X) = x_v^\top B_v\, k_v(X)=\sum_{s=1}^M x_{v,s}p_{v,s} (X).\]

The evolutionary dynamics is given by the the evolutionary game equation on networks (EGN):
\begin{equation}\label{eq:EGN_general}
\dot{x}_{v,s}
= x_{v,s}\bigl(p_{v,s} (X) - \phi_v (X)\bigr),
\end{equation}
for every player $v$ and every strategy $s$.

It follows that a profile $X^*$ is a steady state if and only if, for every player $v$ and every strategy $s$,
\[
x_{v,s}^* > 0 \ \Rightarrow \ p_{v,s} (X^*) = \phi_v (X^*).
\]
In particular, strategies in the support of $x_v^*$ must yield equal payoffs, while no condition is imposed on strategies outside the support.
Therefore, the set of steady states of \eqref{eq:EGN_general} is given by
\[\T=\left\{ X \in \Delta_M^N \;|\; x_{v,s}=0 \;\text{ or } \;\ p_{v,s} (X) = \phi_v (X) ,\; \forall v = 1,\ldots, N,\; \forall s=1,\ldots,M\right\}\]

\medskip

\subsection{Nash equilibria on networks}

A Nash equilibrium is a strategy profile in which no player can improve its payoff by unilaterally deviating from its current mixed strategy. Formally, the set of Nash equilibria is
\[
\TN
=
\left\{
X\in\Delta_M^N
\;\middle|\;
x_v^\top B_v\,k_v(X)
\ge
y^\top B_v\,k_v(X),
\quad
\forall v=1,\dots,N,\ \forall y\in\Delta_M
\right\}.
\]

This condition is equivalent to requiring that, for every player $v$ and all strategies $s,r$,
\[
x_{v,s}>0
\quad\Rightarrow\quad
p_{v,s} (X)\ge p_{v,r} (X).
\]

This implies that, if $X\in\TN$, then for every player $v$
\[
x_{v,s}>0 \ \Rightarrow\ p_{v,s} (X)=\phi_v (X),
\qquad
x_{v,s}=0 \ \Rightarrow\ p_{v,s} (X)\le \phi_v (X).
\]
In particular, every Nash equilibrium is a rest point of the EGN dynamics \eqref{eq:EGN_general}.

\subsection{Boundary Nash equilibria}

We are especially interested in boundary mixed states, that is, strategy profiles in which at least one player plays a pure strategy and at least one player does not. In this setting, it is natural to introduce a weaker Nash-type notion, where the equilibrium condition is imposed only on the players that are not pure. We stress that this notion is not related to the geometric boundary of $\Delta_M^N$, but rather to the presence of both pure and non-pure players in the profile.

Given a profile $X\in \Delta_M^N$, define
\[
\mathcal{M}(X):=\{v\in\{1,\dots,N\}\;|\; x_v \text{ is not pure}\}.
\]

\begin{defi}
A profile $X^*\in \Delta_M^N$ with $|\mathcal{M}(X^*)|<N$ is called a boundary Nash equilibrium if:
\begin{itemize}
\item[i)] $\mathcal{M}(X^*)=\emptyset$, then $X^*\in \TN$;
\item[ii)] $\mathcal{M}(X^*)\neq\emptyset$, then for every player $v\in\mathcal{M}(X^*)$ and all strategies $s,r$,
\[
x_{v,s}^*>0
\quad\Rightarrow\quad
p_{v,s} (X^*)\ge p_{v,r} (X^*),
\]
\end{itemize}
We will denote the set of boundary Nash equilibria by $\TBN$.
\end{defi}

In other words, fully pure profiles are boundary Nash equilibria if and only if they are Nash equilibria. Otherwise, the Nash condition is imposed only on the non-pure players, while pure players are disregarded. This notion can be interpreted as imposing Nash-type optimality conditions only on
a subset of players, namely those adopting mixed strategies.

\begin{remark} \label{remark:equal_payoff}
    Note that is $X^*\in\TBN$ then for every $v\in \M(X^*)$ we must have $p_{v,s} (X^*)= p_{v,r} (X^*)$ for every $r,s$ such that $x_{v,s}^*,\,x_{v,r}^*>0$.
\end{remark}

The following result shows that, although boundary Nash equilibria constitute a weaker notion, they still yield rest points of the dynamics.

\begin{prop}\label{prop:bne_rest_point}
Every boundary Nash equilibrium is a steady state point of the EGN dynamics.
\end{prop}

\begin{proof}
Let $X^*\in\Delta_M^N$ be a boundary Nash equilibrium. We show that
\[
\dot x_{v,s}=0
\qquad
\forall v=1,\dots,N,\quad \forall s=1,\dots,M.
\]

Fix a player $v$. First, suppose that $x_v^*$ is pure. Then there exists $r\in\{1,\dots,M\}$ such that
\[
x_{v,r}^*=1,
\qquad
x_{v,s}^*=0 \quad \text{for all } s\neq r.
\]
Hence
\[
\phi_v (X^*)
=
\sum_{s=1}^M x_{v,s}^*\,p_{v,s} (X^*)
=
p_{v,r} (X^*).
\]
Therefore,
\[
\dot x_{v,r}
=
x_{v,r}^*\bigl(p_{v,r} (X^*)-\phi_v (X^*)\bigr)
=
1\cdot\bigl(p_{v,r} (X^*)-p_{v,r} (X^*)\bigr)=0,
\]
and for every $s\neq r$,
\[
\dot x_{v,s}
=
x_{v,s}^*\bigl(p_{v,s} (X^*)-\phi_v (X^*)\bigr)=0,
\]
since $x_{v,s}^*=0$.

Now suppose that $x_v^*$ is not pure. Since $X^*$ is a boundary Nash equilibrium, we have from Remark \ref{remark:equal_payoff} that all strategies played by player $v$ with positive probability have the same payoff. Denote this common value by $c_v$. Then
\[
\phi_v (X^*)
=
\sum_{s=1}^M x_{v,s}^*\,p_{v,s} (X^*)
=
\sum_{x_{v,s}^*>0} x_{v,s}^*\,c_v
=
c_v.
\]
Hence, if $x_{v,s}^*>0$, then
\[
\dot x_{v,s}
=
x_{v,s}^*\bigl(p_{v,s} (X^*)-\phi_v (X^*)\bigr)
=
x_{v,s}^*(c_v-c_v)=0.
\]
If instead $x_{v,s}^*=0$, then trivially $\dot x_{v,s}=0$.

Since this holds for every player $v$ and every strategy $s$, the profile $X^*$ is a rest point of the dynamics.
\end{proof}

The converse of Proposition~\ref{prop:bne_rest_point} does not hold in general. Indeed, there exist rest points of the EGN dynamics that are not boundary Nash equilibria, as it can can be seen with the following example.

\begin{example}
Consider $N=2$, $M=3$, with adjacency matrix
\[
A=\begin{pmatrix}0&1\\1&0\end{pmatrix},
\qquad
B_1=B_2=
\begin{pmatrix}
1&0&0\\
1&0&0\\
2&0&0
\end{pmatrix}.
\]
Let $X^*=(x_1^*,x_2^*)$, where $x_1^*=(1,0,0)$ and $x_2^*=(\tfrac12,\tfrac12,0)$. Then $x_{1,2}^*=x_{1,3}^*=x_{2,3}^*=0$, and therefore the corresponding components of the vector field vanish trivially. Moreover, strategy $1$ is the only strategy in the support of player $1$, and it satisfies $p_{1,1}(X^*)=\phi_1(X^*)$. For player $2$, the support is $\{1,2\}$, and both strategies satisfy $p_{2,1}(X^*)=p_{2,2}(X^*)=\phi_2(X^*)$. Hence $\dot{x}_{v,s}=0$ for every $v$ and $s$, so $X^*$ is a rest point of the dynamics. 

However, $X^*$ is not a boundary Nash equilibrium, since player $2$ assigns positive probability only to strategies $1$ and $2$, while strategy $3$ yields the strictly larger payoff $p_{2,3}(X^*)=2>\phi_2(X^*)=1$. Thus player $2$ has a profitable unilateral deviation, and the boundary Nash condition fails. 
\end{example}

\medskip

\section{Games with 2 Strategies}
We now turn to the particular case of two strategies. Setting $M=2$, we denote the pure strategies by cooperation ($C$) and defection ($D$). For each player $v$, the mixed strategy vector $x_v=(x_{v,1},x_{v,2})\in\Delta_2$ can be identified with its first coordinate $x_{v,1}\in[0,1]$, since the simplex relation yields $x_{v,2}=1-x_{v,1}$. Thus, $x_{v,1}$ represents the fraction of agents playing $C$, while $1-x_{v,1}$ represents the fraction playing $D$. For notational convenience, we shall use a slight abuse of notation and therefore write $x_v$ instead of $x_{v,1}$.

Accordingly, the global state becomes
\[
\mathbf{x} = (x_1,\ldots,x_N)\in[0,1]^N.
\]

Let $B_v \in \mathbb{R}^{2\times 2}$ be the payoff matrix of player $v$. 
As is standard in evolutionary game theory, we can assume without loss 
of generality that $B_v$ is diagonal, namely
\[
B_v =
\begin{pmatrix}
\sigma_{1,v} & 0 \\
0 & \sigma_{2,v}
\end{pmatrix},
\]
since the dynamics depends only on payoff differences, which are invariant 
under the addition of constants to the columns of the payoff matrix.

For a given profile $\mathbf{x}$, the aggregate strategy of the neighbors 
of player $v$ is
\[
k_v(\mathbf{x}) = \sum_{w=1}^N a_{v,w} 
\begin{pmatrix}
x_w \\ 1-x_w
\end{pmatrix}.
\]

Therefore, since we only need the equations for the first strategy in each vertex, we calculate
\begin{align}
p_{v,1} (\x) - \phi_v (\x) &=\begin{pmatrix}
1-x_v & x_v-1
\end{pmatrix} B_v k_v(\mathbf{x}) \nonumber \\
&= (1-x_v)\sum_{w=1}^Na_{v,w}(\sigma_{v,1}x_w-\sigma_{v,2}(1-x_w))\nonumber \\  
&=(1-x_v)\left((\sigma_{v,1}+\sigma_{v,2})\sum_{w=1}^Na_{v,w}x_w-d_v\sigma_{v,2}  \right). \label{eq:pv-phiv-M2}
\end{align}

Substituting \eqref{eq:pv-phiv-M2} into the general dynamics~\eqref{eq:EGN_general}, we obtain that the cooperation level $x_v$ in each vertex $v$ satisfies
\begin{equation}\label{eq:EGN}
\dot{x}_v
=
x_v(1-x_v) f_v(\mathbf{x}),
\qquad v=1,\ldots,N.
\end{equation}
where 
\begin{equation}\label{eq:f_v}
    f_v(\mathbf{x})=(\sigma_{v,1}+\sigma_{v,2})\sum_{w=1}^Na_{v,w}x_w-d_v\sigma_{v,2}
\qquad v=1,\ldots,N.
\end{equation}

The hypercube $[0,1]^N$ is forward invariant under this dynamics, and each 
coordinate face $x_v=0$ and $x_v=1$ is invariant. 

From \eqref{eq:EGN} we see that the set of steady states can be represented by
\begin{equation}\label{eq:fixed_point}
\T =\left\{\mathbf{x}\in[0,1]^N \;|\;\forall v, \quad x_v=0
\quad\text{or}\quad
x_v=1
\quad\text{or}\quad
f_v(\mathbf{x})=0 \right\}.
\end{equation}

It can be decomposed as
\[
\T=\Tp\cup\Tm\cup\Tbm,\]
where:
\begin{itemize}
\item $\Tp := \{\mathbf{x}\in[0,1]^N : x_v \in \{0,1\}, \ \forall v\}$ is the set of pure states;
\item $\Tm := \{\mathbf{x}\in(0,1)^N : f_v(\mathbf{x})=0,\ \forall v\}$ is the set of fully mixed states;
\item $\Tbm := \T \setminus (\Tp \cup \Tm)$ is the set of boundary mixed steady states.
\end{itemize}

Hence, $\Tbm$ consists of those steady states for which at least one player is pure and at least one player is not pure. The aim of this work is to analyze the structure and stability of the boundary mixed steady states $\Tbm$.

In the present case, the conditions for a state $\x\in[0,1]^N$ to be a Nash equilibrium can be stated in terms of the function $f_v$ given in \eqref{eq:f_v}: the set of Nash equilibria is 
\[
\TN = \left\{\x\in[0,1]^N\;|\;\forall v, \;
(x_v=0 \ \text{and}\ f_v(\mathbf{x})\le 0) \;\mbox{ or }
\;
(x_v=1 \ \text{and}\ f_v(\mathbf{x})\ge 0) \;\mbox{ or }
\;
f_v(\mathbf{x})=0
\right\}\]

In particular, if $\x$ is a Nash equilibrium and $x_v\in(0,1)$ for some $v$, then necessarily $f_v(\mathbf{x})=0$. 
This implies that every fully mixed Nash equilibrium must satisfy
\[
f_v(\mathbf{x})=0
\qquad \forall v .
\]

Analogously, the boundary Nash equililbria can be characterized by 
\begin{equation}\label{eq:boundary_nash_2}
\TBN = \left\{ \x\in[0,1]^N\;|\; \x\in\TN\cap\Tp \;\;\mbox{ or }\;\; f_v(\x)=0, \forall v\in \mathcal{M}(\x)  \right\}
\end{equation}

Note that every Nash equilibrium of the game is a steady state of 
the EGN dynamics. The converse is not true since points in $\Tp$ may not be NE. It is well known that for $M=2, \x \in \Tm$ then $\x \in \TN$. And it is also not true that $ \x \in \Tbm$ will imply $\TN$,  since boundary steady states may fail to satisfy the conditions on $f_v$ in the vertices $v$ where $x_v\in\{0,1\}$. However we can prove that if $ \x \in \Tbm$ then it is a boundary Nash equilibrium.




\begin{thm}
Assume $M=2$. Let $\x^*\in [0,1]^N$ such that $0<|\M(x^*)|<N$. Then $\x^*$ is a steady state for the EGN dynamics if and only if it is a boundary Nash equilibrium.
\end{thm}

\begin{proof}
By hypothesis, $\x^*$ contains at least one pure component and at least one mixed component. 

For every vertex $v$ such that $x_v^*\in\{0,1\}$, the steady-state condition in \eqref{eq:fixed_point} is automatically satisfied. Moreover, the characterization of boundary Nash equilibria given in \eqref{eq:boundary_nash_2} imposes no additional condition on these components.

On the other hand, if $x_v^*\in(0,1)$, then the steady-state condition \eqref{eq:fixed_point} reduces to $f_v(\x)=0$, which is  the same requirement appearing in the characterization of boundary Nash equilibria in \eqref{eq:boundary_nash_2}. Therefore, for every mixed component, the conditions for being a steady state and for being a boundary Nash equilibrium coincide.

It follows that $\x^*$ is a steady state if and only if it satisfies the characterization of a boundary Nash equilibrium in \eqref{eq:boundary_nash_2}.
\end{proof}

\medskip 
\subsection{Boundary mixed steady states}
Let us assume that $\Tbm$ is non empty and fix $\x^*\in\Tbm$. Let us define
\[m=|\M(\x^*)|
.\]
So $m$ is the number of players $v$ playing mixed strategies in the strategy profile $\x^*$. We can relabeled the players under some permutation and obtain that $\x^*$ can be written as
\begin{equation}
 \x^* =( x_1,...,x_{m},x_{m+1},...,x_{N}), \label{eq:pm}  
\end{equation}
\noindent 
where $x_i\in(0,1)$ for $1\leq i \leq m$ and $x_i\in\{0,1\}$ for $m+1\leq i \leq N$. Therefore, $m$ is the number of players playing mixed strategies and let us call $p=N-m$ the number of players playing pure strategies in the profile $\x^*$. Doing such relabeling does not change the spectrum of the adjacency matrix since the new matrix $\Tilde{A}$ is similar to $A$, $\Tilde{A}=PAP^{-1}$, where $P$ is a permutation matrix. Of course, this relabeling depends on the point $\x\in\Tbm$. Throughout this section, we assume that the players have been relabeled in such a way as to have the point $\x^*$ written in this form.


The subset of strategies profiles that contain all points of the form \eqref{eq:pm} will be defined as
$$\mathcal{S}_{m,p}:=(0,1)^{m} \times \{0,1\}^{p}.$$

In order to shorten the notation, we will write a point $\x \in \mathcal{S}_{m,p}$ as $\x=(\x^{\textbf{m}},\x^{\textbf{p}})$, where $\x^{\textbf{m}} \in (0,1)^m$ and $\x^{\textbf{p}} \in \{0,1\}^p$.

Given a point $\x =(x_1,\ldots,x_N)=(\x^{\textbf{m}},\x^{\textbf{p}}) \in \mathcal{S}_{m,p}$, we will define the subset generated by $k$-th player, $k \leq m$, as

\[\mathcal{S}^k_{m,p}(\x):=\Big\{(x_1,\ldots,x_{k-1},t,x_{k+1},\ldots,x_{m},x_{m+1},\ldots,x_N) \,\, \Big| \,\, t \in (0,1) \Big\}
\]
and also
\[\mathcal{S}_{m,p}(\x):=\bigcup_{k=1}^m\mathcal{S}^k_{m,p}(\x)
.\]


Let us call $M$ and $P$ the subgraphs of players who are playing mixed and pure strategies in $\x$, respectively. The adjacency matrix $A$ of the graph is of the form,

\begin{equation}
    A= \begin{pmatrix}
        A_{MM} & A_{MP} \\
        A_{PM}  & A_{PP}
    \end{pmatrix}\label{eq:adj_block}
\end{equation} 

\noindent where 
\begin{itemize}
    \item[] $A_{MM}=\Big[a^{MM}_{v,w}\Big]$, with $a^{MM}_{v,w}=a_{v,w}$ for $1 \leq v,w \leq m$, is the adjacency matrix of the subgraph $M$.

\item[] $A_{MP}=\Big[a^{MP}_{v,w}\Big]$, with $a^{MP}_{v,w}=a_{v,w+m}$ for $1 \leq v \leq m$ and $1 \leq w \leq p$

\item[] $A_{PM}=\Big[a^{PM}_{v,w}\Big]$, with $a^{PM}_{v,w}=a_{v+m,w}$ for $1 \leq v \leq p$ and $1 \leq w \leq m$

\item[]$A_{PP}=\Big[a^{PP}_{v,w}\Big]$, with $a^{PP}_{v,w}=a_{v+m,w+m}$ for $1 \leq v,w \leq p$, is the adjacency matrix of the subgraph $P$.
\end{itemize}

We have that $A_{MP}=(A_{PM})^T$ and these block matrices contains the locations of the edges between players playing mixed strategies and the ones playing pure strategies.

In the next proposition we will prove that given two steady states of $\mathcal{S}_{m,p}$ with the same pure strategies components, then the segment line connecting them will be made of steady states as well.

\begin{prop} Let $\x, \y \in \mathcal{S}_{m,p}$ be such that $\x=(\x^{\textbf{m}},\x^{\textbf{p}}) $, $\y=(\y^{\textbf{m}},\x^{\textbf{p}}) $, with $\x^{\textbf{m}},\y^{\textbf{m}} \in (0,1)^m $ and $\x^{\textbf{p}} \in \{0,1\}^p$. If $\x, \y \in \Tbm$, then $(1-t)\x+t\y \in \Tbm$, $\forall \, t \, \in [0,1]$.
\end{prop}
\begin{proof}
    Suppose that $\x, \y \in \Tbm$ and let $t \in [0,1]$. Clearly, the point ${ \bf z}=(1-t)\x+t\y$ is in $\mathcal{S}_{m,p}$. In order to show that $\z$ is in $\Tbm$, we just need to show that its coordinates satisfy the last condition in \eqref{eq:fixed_point}. For $m+1  \leq v \leq N$, we have $z_v=(1-t)x_{v}+tx_{v}=x_v \in \{ 0,1\}$, by hypothesis. Now, if $1 \leq v \leq m$, then $f_v(\x)=0=f_v(\y)$.
Thus,
\begin{align*}
f_v(\z)&= (\sigma_{1,v}+\sigma_{2,v}) \sum_{w=1}^{N} a_{v,w}z_w-d_v\sigma_{2,v} \\
&=(\sigma_{1,v}+\sigma_{2,v})\sum_{w=1}^{N} a_{v,w}\big((1-t)x_{w}+ty_{w} \big)-(1-t+t)d_v\sigma_{2,v}\\ 
&=(1-t)\left((\sigma_{1,v}+\sigma_{2,v}) \sum_{w=1}^{N} a_{v,w}x_{w}-d_v\sigma_{2,v}\right)+t\left((\sigma_{1,v}+\sigma_{2,v}) \sum_{w=1}^{N} a_{v,w}y_{w}-d_v\sigma_{2,v}\right)\\ 
&=(1-t)f_v(\x)+tf_v(\y)\\&=0
\end{align*}
Therefore, for all $t \in [0,1], t\x + (1-t)\y \in \Tbm.  $
\end{proof}

The next result is a technical result of Linear Algebra. This result might be well-known, but for completeness, we offer a proof of it.

\begin{lemma}\label{lemma:diagonal_matrix}
Let $A$ be symmetric $n \times n$ matrix and $D=\diag(d_i)_{i=1}^n$ a diagonal $n \times n$ matrix with $d_i>0$ for all $i$. The eigenvalues of $DA$ are real numbers and the number of negative, positive and zero eigenvalues are the same as in $A$.
\end{lemma}
 
\begin{proof}
Since $D=\diag(d_i)$ com $d_i>0$ for all $i$, let $H=\diag (\sqrt{d_i})$ e $\widetilde{A}=HAH$. Then
\[DA=HHAHH^{-1}=H\widetilde{A}H^{-1}.\]
Therefore, the eigenvalues of $DA$ and $\widetilde{A}$ are the same. They are real because $\widetilde{A}$ is symmetric. Finally, since $\widetilde{A}$ and $A$ are congruent, by the Silvester's Law of Inertia, they have the same number of positive, negative and zero eigenvalues.
\end{proof}

In the main result, it will be important that the matrix $A_{MM}$ is not identically zero. In order to be able to make this assumption, we will separately study the case where $A_{MM}$ is identically zero. 

Let us define the number of neighbors that cooperates and the number of neighbors that defects of a given player $v$.

\begin{defi} Given a vertices $v$ we define the set of neighbors of $v$ that cooperates as 
$$\mathcal{C}_v =\{ w \in V : a_{vw} \neq 0 \; \text{ and } \; x_w=1 \}.$$
Analogously, the set of neighbors of $v$ that defects is: 
$$\mathcal{D}_v = \{ w \in V : a_{vw} \neq 0 \; \text{ and } \; x_w=0 \}.$$
We also define $N_{v,1}=|\mathcal{C}_v|$ and $N_{v,0}=|\mathcal{D}_v|$.
\end{defi}

Note that if a vertices $v$ is connected with only vertices that play a pure strategy, then $d_v=N_{v,1}+N_{v,0}$.

\begin{lemma}
If a vertex $k \in M$ is isolated in $M$, then
\begin{enumerate}[i)]
    \item $\mathcal{S}^{k}_{m,p}(\x^*) \subset \Tbm$.
    \item $\sigma_{1,k}N_{k,1}=\sigma_{2,k} N_{k,0} $.
\end{enumerate}

\end{lemma}
\begin{proof}
Suppose that a node $k \in M$ is not connected to any other node $w \in M$, then $a_{k,w}=a_{w,k}=0$ for $1 \leq w \leq m$. 

For item $i)$, according to equation \eqref{eq:fixed_point}, we have $f_v(\x^*)=0$, for $1 \leq v \leq m$. Explicitly, using \eqref{eq:pm},
\begin{align*}
  0 =f_v(\x^* ) = (\sigma_{1,v}+\sigma_{2,v})\left(\sum_{w=1}^{m} a_{v,w}x_{w}+\sum_{w=m+1}^{N} a_{v,w}x_{w}\right)  - d_v\sigma_{2,v}
\end{align*}
Note that the first sum above thus does not depend on the component $x_k$ since $a_{v,k}=0$.Thus, for any $1
\leq v \leq m$ we find that any $\z \in \mathcal{S}^{k}_{m,p}(\x^*) $ satisfies $f_v(\z )=f_v(\x^* )=0$, since $\z \in \mathcal{S}^{k}_{m,p}(\x^*)$ and $\x^*$ only differs in the $k$-th component.
For $v \geq m+1$, we already have that $z_v\in \{0,1\}$ for any $\z \in \mathcal{S}^{k}_{m,p}(\x^*)$. Therefore, $\mathcal{S}^{k}_{m,p}(\x^*)\subseteq \Tbm$ by \eqref{eq:fixed_point}.

For item $ii)$, note that
\begin{align*}
0 &= f_k(\x^*) \\
&=(\sigma_{1,k}+\sigma_{2,k})\left(\sum_{w=m+1}^{N} a_{v,w}x_{w}\right)  - d_k\sigma_{2,k}\\  
&=(\sigma_{1,k}+\sigma_{2,k})N_{k,1}  - (N_{k,1}+N_{k,0})\sigma_{2,k}\\  
&=\sigma_{1,k}N_{k,1}-\sigma_{2,k}N_{k,0}
\end{align*}
where we used that $k$ is isolated in $M$ for the second and third equalities.
\end{proof}

The previous result guarantees that if a player who plays mixed strategy is not connected to other players who also plays mixed strategy, then the existence of a boundary mixed steady state implies that it is not an isolated steady state.



In the case where the graph of players playing mixed strategies is totally disconnected, we have:

\begin{coro}\label{cor.AMM=0}
If $A_{MM}=0$, then one of the following holds:
\begin{enumerate}
    \item $\Tbm = \emptyset $
    \item For all $\x\in \Tbm$, we have $\mathcal{S}_{m,p}(\x) \subset \Tbm$ and 
    $ \sigma_{1,v}N_{v,1}=\sigma_{2,v}N_{v,0}$, $\forall v \in M$.
\end{enumerate}
\end{coro}

\begin{proof}
    If there is no $\x \in \mathcal{S}_{m,p} $ that is boundary mixed steady state, then ${\emph i)}$ is satisfied. 
    
    If there is $\x \in \mathcal{S}_{m,p}$ that is in steady state, then the previous lemma guarantees that for all vertices $ 1 \leq k \leq m$, $\mathcal{S}^k_{m,p} (\x)\subset \Tbm $, since $A_{MM}=0$ means that all vertices are isolated in $M$. Then, by definition of $\mathcal{S}_{m,p}(\x)$ the inclusion follows. The last identity in $ii)$ also follows from the previous lemma.
\end{proof}

\begin{thm} \label{thm:main} Let $\x \in \Tbm$ be a steady state of the Evolutionary Game on Networks equation
\begin{align}
 \dot{x}_v = x_v(1-x_v)f_v(\x), \qquad v \in V. \label{eq:egn_thm}
\end{align}
Then $\x$ is non-hyperbolic or a saddle point.

In particular, if the adjacency matrix of the subgraph os mixed players satisfies $A_{MM} \neq 0$, and the quantities $\sigma_{v,1}+\sigma_{v,1}$ are nonzero and have the same sign, for all $v\in\M(\x)$, then the Jacobian matrix $J(\x)$ has at least one positive and one negative eigenvalue.
\end{thm}

\begin{proof}
  Given a solution ${ \bf x} \in \Tbm$ of the EGN given in \eqref{eq:egn_thm}, we compute the Jacobian $J(\x)$ of the system at $\x$ and obtain:
  \begin{equation}
  J_{vw}(\x)=   \begin{cases}
			(1-2x_v)f_v(\x) &, \text{if} \,\, v=w\\
              x_v(1-x_v)(\sigma_{1,v}+\sigma_{2,v}) a_{vw} &, \text{if} \,\, v \neq w.\\
		 \end{cases}    
  \end{equation}

Let us see the Jacobian matrix $J$ as the block matrix given by
\begin{equation}
    J(\x)=\begin{pmatrix}
    J_{MM}(\x) & J_{MP}(\x) \\
    J_{PM}(\x) & J_{PP}(\x)\\        
    \end{pmatrix}  
 \label{eq:jac_block}
\end{equation} 
and proceed to analyze these block matrices.

For $J_{PM}(\x)$, note that $v>m$ and $w \leq m$, implies that 
\[J_{vw}(\x)=x_v(1-x_v)(\sigma_{1,v}+\sigma_{2,v})a_{vw}=0,\] 
since $x_v \in \{ 0,1\}$ for $v>m$. Thus $J_{PM}(\z)$ is the zero matrix of dimension $p \times m$. 

In order to save some space, let us define $\gamma_v=x_v(1-x_v)(\sigma_{1,v}+\sigma_{2,v})$.

For $J_{MM}(\x)$ we have $v,w \leq m$ and then $f_v(\x)=0$ and 
\begin{equation}
  J_{vw}(\x)=   \begin{cases}
			0, & \text{if} \,\, v=w\\
               \gamma_v a_{vw}, & \text{if} \,\, v \neq w.
		 \end{cases}    
  \end{equation}
Thus $J_{MM}(\x)=DA_{MM}$ where $D=\diag(\gamma_v)_{v=1}^m$. 

In $J_{MP}(\x)$ we have $ v\leq m$ and $ w > m$, and therefore  
\[J_{vw}(\x)=\gamma_v a_{vw}.\] 
Since  $x_v \in (0,1)$, thus $J_{vw}(\x)$ does not need to be zero. But we can notice that $J_{MP}(\x)=DA_{MP}$ 
   
Finally, for $J_{PP}(\x)$ we have $v,w>m$ and we again have that  $x_v \in \{0,1\}$, thus 
      \begin{equation}
  J_{vw}(\x)=   \begin{cases}
			(1-2x_v)f_v(\x), & \text{if} \,\, v=w\\
              0, & \text{if} \,\, v \neq w\\
		 \end{cases}    
  \end{equation}
i.e., $J_{PP}(\x)$ is a diagonal matrix.

Therefore, the Jacobian is of the form
\begin{equation}
    J(\x)= 
    \begin{pmatrix}
        DA_{MM} & D A_{MP} \\
        0       & J_{PP}(\x) \\
    \end{pmatrix}
 \label{eq:jac_block_triangular}
\end{equation} 
and it follows that its eigenvalues are the eigenvalues of $DA_{MM}$ and $J_{PP}(\x)$. Since we have $a_{vv}=0$ for all $v$, it follows that
\[\Tr(DA_{MM})=\sum_{i=1}^N\sum_{j=1}^Nd_{ij}a_{ji} = \sum_{i=1}^N d_{ii}a_{ii}=0.\]
Therefore, the eigenvalues of $DA_{MM}$ are all zero or there must be at least two eigenvalues of oposite signs. In either case, $\x$ is non-hyperbolic or a saddle point.

Now, if for all $v\le m$, $\sigma_{1,v}+\sigma_{2,v}$ are nonzero and have all the same sign, then $\gamma_v$ has the same sign also, since $x_v\in(0,1)$.

Since the matrix $A_{MM}$ symmetric and non zero, all eigenvalues are real and Perron-Frobenius guarantees a positive eigenvalue. However since the trace of $A_{MM}$ is zero, we also must have a negative eigenvalue. Thus, Lemma \ref{lemma:diagonal_matrix} implies that the matrix $DA_{MM}$ (or $-DA_{MM}$) has the same number of positive, negative and zero eigenvalues as $A_{MM}$. Since $J_{PP}(\x)$ is diagonal, we conclude that all eigenvalues of $J(\x)$ are real, and there is at least one which is positive and one negative. So $\x$ is a saddle point.
 \end{proof}

In particular, boundary mixed equilibria are never asymptotically stable. This contrasts with the well-mixed setting, where boundary mixed steady  states may be asymptotically stable (for instance, in games such as the Hawk--Dove game). The following corollaries provide conditions under which a steady state $\x \in \Tbm$ is necessarily non-hyperbolic, and therefore cannot be a saddle point.

\begin{coro}
Let $\x \in \Tbm$ be a steady state of the EGN. If the matrix $A_{MM}$ is singular, then $\x$ is a non-hyperbolic steady state.
\end{coro}

\begin{proof}
From the proof of Theorem \ref{thm:main}, we have that the Jacobian matrix can be written as
\[
J(\x)=\begin{pmatrix}
DA_{MM} & DA_{MP}\\
0 & J_{PP}(\x)
\end{pmatrix},
\]
where $D=\diag(\gamma_v)_{v=1}^m$ with $\gamma_v \neq 0$ for all $v \in M$, since $x_v \in (0,1)$ and $\sigma_{1,v}+\sigma_{2,v} \neq 0$.

Since $D$ is invertible, it follows that
\[
\ker(DA_{MM}) = \ker(A_{MM}).
\]
Thus, if $A_{MM}$ is singular, then $\ker(A_{MM}) \neq \{0\}$, which implies that $0$ is an eigenvalue of $DA_{MM}$.

Since the eigenvalues of $J(\x)$ are the union of the eigenvalues of $DA_{MM}$ and $J_{PP}(\x)$, it follows that $0 \in \sigma(J(\x))$. Therefore, $\x$ is a non-hyperbolic steady state.
\end{proof}

\begin{coro}
Let $\x \in \Tbm$ be a steady state of the EGN. If $|\M| > \rank(A)$, then $\x$ is a non-hyperbolic steady state.
\end{coro}

\begin{proof}
Recall that the rank of a matrix is the largest order of a nonzero minor. 
If $A_{MM}$ were invertible, then $\det(A_{MM}) \neq 0$, and therefore $A$ 
would have a nonzero minor of order $|M|$. This would imply that 
$\rank(A) \ge |M|$, which contradicts the assumption $|M| > \rank(A)$. 

Thus, $A_{MM}$ is singular, and the result follows from the previous corollary.
\end{proof}


\noindent
In the degenerate case where mixed players do not interact, the equilibrium is non-hyperbolic and may exhibit neutral behavior. Otherwise, whenever there is interaction among mixed players, the Jacobian has eigenvalues with both positive and negative real parts, implying that the equilibrium is unstable. In this case, the Jacobian may also have zero eigenvalues, so the equilibrium is not necessarily a saddle. This is illustrated in the following example.

\begin{example}
Consider the graph $\mathcal G$ given by the tree with two leaves, where vertex $1$ is the central node and vertices $2$ and $3$ are the leaves. We present three simple configurations illustrating the possible behavior of steady states and of the spectrum of the Jacobian.

Let
\[
I:=\begin{pmatrix}1&0\\0&1\end{pmatrix},
\qquad
Q:=\begin{pmatrix}3&0\\0&1\end{pmatrix}.
\]

\begin{enumerate}[a)]

    \item If $B_1=B_2=B_3=I,$ then, for every $m\in[0,1]$, $\x^*=(m,\,1,\,0)$ is a steady state, with $\sigma\bigl(J(\x^*)\bigr)=\{0,\,1-2m,\,2m-1\}$. In particular, if $m=\frac12$, then $\sigma\bigl(J(\x^*)\bigr)=\{0\}$. Also, $\x^*$ is a Nash equilibrium if and only if $m=\frac12$.

    \item If $B_1=Q,$ and $B_2=B_3=I,$ then $\x^*=\left(\frac12,\,\frac12,\,0\right)$ is a boundary mixed steady state with a zero eigenvalue and two non-zero eigenvalues. $\x^*$ is a Nash equilibrium.

    \item If $B_1=B_3=Q,$ and $B_2=I,$ then $\x^*=\left(\frac12,\,\frac12,\,0\right)$
    is a boundary mixed steady state and it is a saddle point. The point $\x^*$ in this case is not a Nash equilibrium.
\end{enumerate}

These three situations are summarized in Table~\ref{tab:star-summary}.

\begin{table}[H]
\centering
\caption{Examples of steady states for a tree with 2 leaves.}
\label{tab:star-summary}
\renewcommand{\arraystretch}{1.2}
\setlength{\tabcolsep}{10pt}
\begin{tabular}{c c c c}
\toprule
\begin{tabular}[c]{@{}c@{}}Payoff matrices\\ $(B_1,B_2,B_3)$\end{tabular}
&  \begin{tabular}[c]{@{}c@{}} Steady state $\x^*$\\ $(x_1,x_2,x_3)$\end{tabular}   & Spectrum $\sigma(J(\x^*))$ & Figure \\
\midrule

$\left(I,I,I\right)$
&
$\left(\frac12,\,1,\,0\right)$
&
$\{0,\,0,\,0\}$
&
\staricon{mixed}{coop}{defe}
\\[5ex]

$\left(Q,I,I\right)$
&
$\left(\frac12,\,\frac12,\;0\,\right)$
&
$\left\{-\frac{1}{\sqrt2},\,0,\,\frac{1}{\sqrt2}\right\}$
&
\staricon{mixed}{mixed}{defe}
\\[5ex]

$\left(Q,I,Q\right)$
&
$\left(\frac12,\,\frac12,\,0\right)$
&
$\left\{-\frac1{\sqrt2},\,\frac1{\sqrt2},\,1\right\}$
&
\staricon{mixed}{mixed}{defe}\\
\bottomrule
\end{tabular}
\end{table}

\end{example}

\section{Conclusions}

In this paper, we analyzed the structure and stability of boundary mixed steady states in evolutionary games on networks, that is, equilibria in which, in a finite set of players, some nodes correspond to pure strategies while others correspond to mixed strategies. 
We introduced the notion of boundary Nash equilibrium, a natural relaxation of the classical Nash equilibrium in which the optimality condition is imposed only on non-pure players. In the case of two strategies, we showed that this notion provides a complete characterization of boundary mixed steady states, while this equivalence fails in higher dimensions.

From a structural perspective, we identified a degenerate regime arising when mixed players do not interact. In this case, boundary mixed steady states are not isolated: whenever such an equilibrium exists, it belongs to a continuum of steady states obtained by freely varying the strategies of isolated mixed players. This reveals a loss of rigidity in the dynamics, where equilibrium configurations may form entire segments rather than discrete points.

Our main results concern the stability properties of these equilibria. Under mild assumptions on the payoff matrices, we proved that whenever the subgraph induced by mixed players is nontrivial, the Jacobian at a boundary mixed steady state necessarily admits eigenvalues with both positive and negative real parts. Consequently, such equilibria are unstable. 

Therefore, when mixed players interact, instability arises through the coexistence of expanding and contracting directions; when they do not interact sufficiently, the equilibrium becomes non-hyperbolic and belongs to a non-isolated set of steady states.

These results highlight a fundamental contrast between classical well-mixed replicator dynamics and their networked counterparts. In the well-mixed setting, the population is homogeneous and described by a single aggregate strategy, so heterogeneous configurations in which pure and mixed behaviors coexist across different agents are not allowed. In contrast, network structure naturally allows for such heterogeneous profiles, but imposes strong constraints on their stability. In this sense, network structure leaves no room for robust mixed–pure coexistence: boundary mixed equilibria can only manifest as unstable saddles or as degenerate configurations.

A natural direction for future research is to investigate whether additional mechanisms, such as self-regulation or adaptive network structures, may stabilize boundary mixed steady states. In this context, our results reveal a structural limitation of evolutionary dynamics on networks: under natural assumptions on the payoff matrices, boundary mixed steady states are never asymptotically stable, as the Jacobian generically exhibits both positive and negative directions or degeneracies. This instability is not merely a consequence of specific payoff parameters, but rather an intrinsic feature of the network structure. This perspective connects naturally with recent approaches introducing control mechanisms in evolutionary games on networks in order to stabilize the dynamics. Exogenous control, as in \cite{RiehlCao2017}, shows that a suitably chosen subset of agents can steer the system toward a desired homogeneous state. In contrast, the self-regulated EGN model in \cite{MadeoMocenni2021} introduces endogenous feedback, effectively modifying the local dynamics and enabling the stabilization of mixed consensus states, even with global convergence.

\newpage
\bibliographystyle{plainnat}
\bibliography{references}

@phdthesis{Carvalho1983,
  author      = {De Carvalho, Maria Samuel Bezerra},
  title       = {Dynamical Systems and Game Theory},
  school      = {University of Warwick},
  year        = {1983},
  address     = {Coventry, UK},
  url         = {https://wrap.warwick.ac.uk/111054/},
  note        = {Doctoral dissertation}
}

@book {Hofbauer_sigmund,
    AUTHOR = {Hofbauer, Josef and Sigmund, Karl},
     TITLE = {Evolutionary games and population dynamics},
 PUBLISHER = {Cambridge University Press, Cambridge},
      YEAR = {1998},
     PAGES = {xxviii+323},
      ISBN = {0-521-62365-0; 0-521-62570-X},
   MRCLASS = {92D25 (34C35 34D99 90-02 92-02 92D40)},
  MRNUMBER = {1635735},
MRREVIEWER = {Gabriela\ Schranz-Kirlinger},
       DOI = {10.1017/CBO9781139173179},
       URL = {https://doi.org/10.1017/CBO9781139173179},
}

@incollection{JaZe15,
title = {Chapter 3 - Games on Networks},
editor = {H. Peyton Young and Shmuel Zamir},
series = {Handbook of Game Theory with Economic Applications},
publisher = {Elsevier},
volume = {4},
pages = {95-163},
year = {2015},
issn = {1574-0005},
doi = {https://doi.org/10.1016/B978-0-444-53766-9.00003-3},
url = {https://www.sciencedirect.com/science/article/pii/B9780444537669000033},
author = {Matthew O. Jackson and Yves Zenou}
}

@article {Kitching2025,
    AUTHOR = {Kitching, Christopher R. and Ramirez, Luc\'ia S. and San Miguel, Maxi and Galla, Tobias},
     TITLE = {Breaking coexistence: zealotry vs nonlinear social impact},
   JOURNAL = {Chaos},
  FJOURNAL = {Chaos. An Interdisciplinary Journal of Nonlinear Science},
    VOLUME = {35},
      YEAR = {2025},
    NUMBER = {8},
     PAGES = {Paper No. 083133, 25},
      ISSN = {1054-1500,1089-7682},
   MRCLASS = {91A22 (37N40 91B12 91D30)},
  MRNUMBER = {4948276},
       DOI = {10.1063/5.0282676},
       URL = {https://doi.org/10.1063/5.0282676},
}

@article{MaWeFu14,
    doi = {10.1371/journal.pcbi.1003567},
    author = {Maciejewski, Wes AND Fu, Feng AND Hauert, Christoph},
    journal = {PLOS Computational Biology},
    publisher = {Public Library of Science},
    title = {Evolutionary Game Dynamics in Populations with Heterogenous Structures},
    year = {2014},
    month = {04},
    volume = {10},
    url = {https://doi.org/10.1371/journal.pcbi.1003567},
    pages = {1-16},
    number = {4},

}

@article {MaMo15,
    AUTHOR = {Madeo, Dario and Mocenni, Chiara},
     TITLE = {Game interactions and dynamics on networked populations},
   JOURNAL = {IEEE Trans. Automat. Control},
  FJOURNAL = {Institute of Electrical and Electronics Engineers.
              Transactions on Automatic Control},
    VOLUME = {60},
      YEAR = {2015},
    NUMBER = {7},
     PAGES = {1801--1810},
      ISSN = {0018-9286,1558-2523},
   MRCLASS = {91A25 (91D30)},
  MRNUMBER = {3365069},
       DOI = {10.1109/TAC.2014.2384755},
       URL = {https://doi.org/10.1109/TAC.2014.2384755},
}

@article {MadeoMocenni2021,
    AUTHOR = {Madeo, Dario and Mocenni, Chiara},
     TITLE = {Consensus towards partially cooperative strategies in
              self-regulated evolutionary games on networks},
   JOURNAL = {Games},
  FJOURNAL = {Games},
    VOLUME = {12},
      YEAR = {2021},
    NUMBER = {3},
     PAGES = {Paper No. 60, 16},
      ISSN = {2073-4336},
   MRCLASS = {91A22 (91A43)},
  MRNUMBER = {4323656},
       DOI = {10.3390/g12030060},
       URL = {https://doi.org/10.3390/g12030060},
}

@article {MoMo24,
    AUTHOR = {Mocenni, Chiara and Moraes, Jean Carlo},
     TITLE = {Pure {N}ash equilibrium and independent dominating sets in
              evolutionary games on networks},
   JOURNAL = {J. Dyn. Games},
  FJOURNAL = {Journal of Dynamics and Games},
    VOLUME = {11},
      YEAR = {2024},
    NUMBER = {3},
     PAGES = {280--294},
      ISSN = {2164-6066,2164-6074},
   MRCLASS = {91A22 (91A43)},
  MRNUMBER = {4737069},
       DOI = {10.3934/jdg.2023027},
       URL = {https://doi.org/10.3934/jdg.2023027},
}

@article{ON2006,
  author  = {Ohtsuki, Hisashi and Hauert, Christoph and Lieberman, Erez and Nowak, Martin A.},
  title   = {A simple rule for the evolution of cooperation on graphs and social networks},
  journal = {Nature},
  year    = {2006},
  month   = {May},
  volume  = {441},
  number  = {7092},
  pages   = {502--505},
  doi     = {10.1038/nature04605},
  issn    = {1476-4687},
  url     = {https://doi.org/10.1038/nature04605}
}

@article {OhNo06,
    AUTHOR = {Ohtsuki, Hisashi and Nowak, Martin A.},
     TITLE = {The replicator equation on graphs},
   JOURNAL = {J. Theoret. Biol.},
  FJOURNAL = {Journal of Theoretical Biology},
    VOLUME = {243},
      YEAR = {2006},
    NUMBER = {1},
     PAGES = {86--97},
      ISSN = {0022-5193,1095-8541},
   MRCLASS = {92D15 (05C90 91A22 91A43)},
  MRNUMBER = {2279323},
       DOI = {10.1016/j.jtbi.2006.06.004},
       URL = {https://doi.org/10.1016/j.jtbi.2006.06.004},
}

@article {RiehlCao2017,
    AUTHOR = {Riehl, James R. and Cao, Ming},
     TITLE = {Towards optimal control of evolutionary games on networks},
   JOURNAL = {IEEE Trans. Automat. Control},
  FJOURNAL = {Institute of Electrical and Electronics Engineers.
              Transactions on Automatic Control},
    VOLUME = {62},
      YEAR = {2017},
    NUMBER = {1},
     PAGES = {458--462},
      ISSN = {0018-9286,1558-2523},
   MRCLASS = {91A22},
  MRNUMBER = {3598029},
       DOI = {10.1109/TAC.2016.2558290},
       URL = {https://doi.org/10.1109/TAC.2016.2558290},
}

@article{SmPr73,
  author  = {Maynard Smith, J. and Price, G. R.},
  title   = {The Logic of Animal Conflict},
  journal = {Nature},
  year    = {1973},
  volume  = {246},
  number  = {5427},
  pages   = {15--18},
  doi     = {10.1038/246015a0},
  url     = {https://doi.org/10.1038/246015a0}
}

@book{smith_1982,
  author    = {Maynard Smith, John},
  title     = {Evolution and the Theory of Games},
  publisher = {Cambridge University Press},
  year      = {1982},
  address   = {Cambridge},
  isbn      = {9780521286923},
  doi       = {10.1017/CBO9780511806292}
}

@article {Szolnoki2016,
    AUTHOR = {Szolnoki, Attila and Perc, Matja\v z},
     TITLE = {Zealots tame oscillations in the spatial rock-paper-scissors
              game},
   JOURNAL = {Phys. Rev. E},
  FJOURNAL = {Physical Review E},
    VOLUME = {93},
      YEAR = {2016},
    NUMBER = {6},
     PAGES = {062307, 6},
      ISSN = {2470-0045,2470-0053},
   MRCLASS = {91A22},
  MRNUMBER = {3714380},
       DOI = {10.1103/physreve.93.062307},
       URL = {https://doi.org/10.1103/physreve.93.062307},
}

@Article{TAJO78,
  author  = {Taylor, P. D. and Jonker, L. B.},
  title   = {Evolutionary stable strategies and game dynamics},
  journal = {Mathematical Biosciences},
  year    = {1978},
  volume  = {40},
  pages   = {145--156},
  doi     = {10.1016/0025-5564(78)90077-9},
  url     = {https://doi.org/10.1016/0025-5564(78)90077-9}
}

@InProceedings{Zeeman1980,
author="Zeeman, E. C.",
editor="Nitecki, Zbigniew
and Robinson, Clark",
title="Population dynamics from game theory",
booktitle="Global Theory of Dynamical Systems",
year="1980",
publisher="Springer Berlin Heidelberg",
address="Berlin, Heidelberg",
pages="471--497",
isbn="978-3-540-38312-3",
DOI = {10.1063/5.0282676},
URL = {https://doi.org/10.1063/5.0282676},
}

@book{weibull1995,
  author    = {Weibull, J{\"o}rgen W.},
  title     = {Evolutionary Game Theory},
  publisher = {MIT Press},
  year      = {1995},
  address   = {Cambridge, MA},
  isbn      = {9780262231817},
  url       = {https://mitpress.mit.edu/9780262231817/evolutionary-game-theory/}
}

\end{document}